\documentclass{birkjour}
\usepackage{amssymb}
\usepackage{amsthm}
\usepackage{amsmath}
\usepackage{stmaryrd}
\usepackage{array}

\usepackage{enumitem}

\newcommand{\norm}[1]{\left\lVert#1\right\rVert}

\renewcommand{\labelenumi}{\theenumi}
\renewcommand{\theenumi}{{\it(\roman{enumi})}}

 \newtheorem{theorem}{Theorem}[section]
 \newtheorem{corollary}[theorem]{Corollary}
 \newtheorem{lemma}[theorem]{Lemma}
 \newtheorem{proposition}[theorem]{Proposition}
 \theoremstyle{definition}
 \newtheorem{definition}[theorem]{Definition}
 \theoremstyle{remark}
 \newtheorem{remark}[theorem]{Remark}
 
 \numberwithin{equation}{section}

\def\N{{\mathbb N}}
\def\F{{\mathbb F}}
\def\Z{{\mathbb Z}}
\def\C{{\mathbb C}}
\def\R{{\mathbb R}}

\def\Span{{\rm Span}\ }

\def\PSL2{PSL_2 (\mathbb{Z})}
\def\SL2{SL_2 (\mathbb{Z})}

\def\lbrack{\left[}
\def\rbrack{\right]}

\def\ad{{\rm ad}\ }

\def\indset{\mathcal{I}}

\def\indsetL4{\indset^*}

\def\Basis5{\mathcal{B}}
\def\Reps5{\mathcal{S}}

\def\freeassoc3{\F\left< A,B,C\right>}

\def\algebH{\mathcal{H}}
\def\Heisen{\algebH(q)}

\def\HLie{\mathfrak{L}}
\def\HeisenLie{\HLie(q)}

\def\preLie{\HLie_0(q)}
\def\approxLie{\mathfrak{E}(q)}
\def\derpreLie{\HLie_0^{(2)}(q)}

\def\LieAB{\lbrack A,B\rbrack}
\def\algI{I}

\def\baseG{\gamma}

\def\HilbSeq{\ell_2}
\def\Hilb{\mathfrak{H}}
\def\HilbD{\Hilb_0}

\def\basvecHilb{\Phi}

\def\newHA{a}
\def\newHB{a^{+}}

\def\preHeisen{\mathcal{K}}

\def\OpAlg{\mathfrak{B}(\Hilb)}

\def\commdiag{\omega_n}
\def\commdiagS{\left(\commdiag\right)}

\def\lket{\left|}
\def\rket{\right>}



\begin{document}

%
%
%
%
%
%
%
%
%

\title[A $q$-deformed Heisenberg algebra as a normed space]
 {A $q$-deformed Heisenberg algebra as a normed space}

\author[Rafael Reno S. Cantuba]{Rafael Reno S. Cantuba}

\address{%
Mathematics and Statistics Department\\
De La Salle University, Manila\\
2401 Taft Ave., Malate, Manila\\
1004 Metro Manila\\
Philippines}

\email{rafael\_cantuba@dlsu.edu.ph}




\date{June 13, 2019}

\begin{abstract}
Given a real number $q$ such that $0<q<1$, the natural setting for the mathematics of a $q$-oscillator is an infinite-dimensional, separable Hilbert space that is said to provide an interpolation between the Bargmann-Segal space of holomorphic functions and the Hardy-Lebesgue space of analytic functions. The traditional basis states are interrelated by the creation and annihilation operators. Since the commutation relation is $q$-deformed, the commutator algebra for the creation and annihilation operators is not a low-dimensional Lie algebra like that for the canonical commutation relation. In this study, a characterization of the elements of the said commutator algebra is obtained using spectral properties of the creation and annihilation operators as these faithfully represent the generators of a $q$-deformed Heisenberg algebra. The derived algebra of the commutator algebra is precisely the set of all compact operators, and the resulting Calkin algebra is algebraically isomorphic to the complex algebra of Laurent polynomials in one indeterminate. As for any operator that is not in the commutator algebra, the action of such an operator on an arbitrary basis state can be approximated by a Lie series of elements from the commutator algebra. 
\end{abstract}

\maketitle
\section{Introduction} Consider a real number $q\in\left]0,1\right[$, and let $\Hilb=\Hilb_q$ denote an infinite dimensional separable Hilbert space over the complex field $\C$. This Hilbert space depends on the parameter $q$ in the sense that if we fix a complete orthonormal basis for some dense subset of $\Hilb$, which consists of the usual ket vectors $\lket n\rket$, where $n$ ranges over all nonnegative integers, then such basis elements are generated from the basis element $\lket 0\rket$ by the \emph{creation operator} $\newHB$, (i.e., $\newHB\lket n\rket =c_n\lket n+1\rket $ for some nonzero scalar $c_n$ for any $n$), and that the operator $\newHB$ has an adjoint $\newHA$ called the \emph{annihilation operator}, and that these two operators satisfy the \emph{$q$-deformed commutation relation}

\begin{eqnarray}\label{qComm}
\newHA\newHB-q\newHB\newHA = 1.
\end{eqnarray}
The Hilbert space $\Hilb$ is the natural setting for the mathematics of the \emph{$q$-deformed quantum harmonic oscillator} or simply \emph{$q$-oscillator}, about which the literature is vast. We mention but only a few seminal ones. One of the main theoretical developments is the perspective that $\Hilb$ provides a rather
remarkable interpolation between the Bargmann-Segal
space and the Hardy-Lebesgue space. This was developed in one of the early seminal papers about $q$-oscillators which is \cite{Ari76}. By a suitable choice of an inner product and by a limiting process such that $q\rightarrow 1$, the Hilbert space $\Hilb$ becomes the Bargmann-Segal space of holomorphic functions in several complex variables that satisfy a square-integrability condition with respect to some Lebesgue measure. If $q\rightarrow 0$, by some other choice of an inner product, $\Hilb$ becomes the Hardy-Lebesgue space of all analytic functions in the unit disk.

Results independently developed by Macfarlane \cite{Mac89} and Biedenharn \cite{Bie89} are also seminal in the theoretical development of the $q$-oscillator, hence the so-called ``Macfarlane-Biedenharn $q$-oscillator," from which other $q$-oscillators were formulated. Among many physical and mathematical applications, the significance of the $q$-oscillator is that, together with the deformed quantum group $SU_q(2)$, these two mathematical constructions were once considered as possible candidates for the role of oscillator and angular momentum, respectively, at very small scales, such as the Planck scale \cite[p. 1155]{Zhe91}. 

Since $q\neq 1$, the left-hand side of \eqref{qComm} is not the commutator $\lbrack \newHA,\newHB\rbrack:=\newHA\newHB-\newHB\newHA$, and furthermore this operator and its positive powers have nontrivial action on the ket basis vectors, as we shall later see. The motivation for this work is the simple question of what is the smallest linear space of operators that contains the creation and annihilation operators and is closed under the commutation operation on operators. This linear space consists of linear combinations of nested commutators in the operators $\newHA$ and $\newHB$. Also, this linear space may also be termed the \emph{commutator algebra} of the $q$-oscillator, and we call the said linear combinations as \emph{Lie polynomials} in $\newHA$, $\newHB$. Let us call the problem of characterizing the Lie polynomials in $\newHA$, $\newHB$ as a \emph{Lie polynomial characterization problem} for the algebra generated by $\newHA$, $\newHB$.

The solution to this kind of a problem turns out to be rather simple if the commutation relation were not $q$-deformed. From the canonical commutation relation
\begin{eqnarray}
\lbrack\newHA,\newHB\rbrack = 1,\label{classicComm}
\end{eqnarray}
the operators $\newHA$ and $\newHB$ generate the \emph{Heisenberg-Weyl algebra}, which has the property that, using results from the theory of the classification of low-dimensional Lie algebras, the set of all Lie polynomials in $\newHA$, $\newHB$ is the \emph{Heisenberg (Lie) algebra}, which is the three-dimensional Lie algebra whose derived (Lie) algebra is contained in the center. The operators $\newHB$, $\newHA$ and $1$ form a basis for the commutator algebra \cite[Section 3.2.1]{Erd06}.

Since the Lie polynomial characterization problem for the creation and annihilation operators of the $q$-oscillator is nontrivial, we shall need an abstraction of the algebra generated by the creation and annihilation operators. This is the $q$-deformed Heisenberg algebra, a mathematical abstraction of the $q$-deformed commutation relation. Most of the algebraic aspects of this algebra have been well-studied \cite{Hel,Hel05,Lau}. The $q$-deformed Heisenberg algebra has a wide range of applications and relationships to various areas of mathematics \cite[Section 1.2]{Hel}. In this work, we show how Operator Theory, and in particular, the spectral characteristics of a unilateral weighted shift can be used to solve the said Lie polynomial characterization problem. In particular, we show that compact operators generated by $\newHA$, $\newHB$ is precisely the derived (Lie) algebra of the set of all Lie polynomials in $\newHA$, $\newHB$. We also obtain a result which states that the resulting Calkin algebra (the quotient algebra modulo the ideal of all compact operators) is algebraically isomorphic to the algebra over the complex field of all Laurent polynomials in one indeterminate. If an operator generated from the creation and annihilation operators is not in the commutator algebra, we have a result which states that the action of such an operator on an arbitrary basis element $\lket n\rket$ of $\Hilb$ can be approximated by a Lie series, or a power series of Lie polynomials, in $\newHA$, $\newHB$. 

The algebraic framework is laid out in Section~\ref{prelSec}. The involvement of Operator Theory starts in Section~\ref{HilbRepSec}, where we establish that under certain conditions, the $q$-deformed Heisenberg algebra is an algebra of Hilbert space operators, and we determine the resulting Calkin algebra. Finally, in Section~\ref{ApproxSec}, we identify which linear subspace of the $q$-deformed Heisenberg algebra is precisely the set of all Lie polynomials in the creation and annihilation operators. Elements of the resulting derived algebra, which we show to be consisting of those linear combinations of nested commutators without a term consisting of a generator, can be characterized by compactness.

\section{Preliminaries}\label{prelSec}

Let $\F$ be a field, and let $\mathcal{A}$ be a unital associative algebra over $\F$. We reserve the term \emph{subalgebra} to refer to a subset of $\mathcal{A}$ that is also a unital associative algebra over $\F$ using the same operations. The term subalgebra is to occur in a couple of other contexts in this work, and so we clarify the usage of this term in the following. We turn $\mathcal{A}$ into a Lie algebra over $\F$ with Lie bracket given by $\lbrack f,g\rbrack := fg-gf$ for all $f,g\in\mathcal{A}$. So that the algebraic substructure under the associative algebra structure in $\mathcal{A}$ is distinguished from the  substructure under the Lie algebra structure, we use the term \emph{Lie subalgebra} to mean a subset of $\mathcal{A}$ that is also a Lie algebra over $\F$ under the same Lie algebra operations. We treat the terms \emph{ideal} and \emph{Lie ideal} similarly. 

\begin{definition}\label{LiegenDef}
Given $f_1,f_2,\ldots,f_n\in\mathcal{A}$, we define the \emph{Lie subalgebra of $\mathcal{A}$ generated by $f_1,f_2,\ldots,f_n$} as the smallest Lie subalgebra $\mathcal{B}$ of $\mathcal{A}$ that contains $f_1,f_2,\ldots,f_n$; i.e., if $\mathcal{C}$ is a Lie subalgebra of $\mathcal{A}$, with $f_1,f_2,\ldots,f_n\in\mathcal{C}$ and that $\mathcal{C}$ is contained in $\mathcal{B}$, then $\mathcal{B}=\mathcal{C}$.  The elements of the Lie subalgebra of $\mathcal{A}$ generated by $f_1,f_2,\ldots,f_n$ are called the \emph{Lie polynomials} in $f_1,f_2,\ldots,f_n$. Given an index set $J$, a formal sum $\sum_{m\in J} X_m$ where $X_m$ is a Lie polynomial in $f_1,f_2,\ldots,f_n$ for all $m$ is called a \emph{Lie series} in $f_1,f_2,\ldots,f_n$. 
\end{definition}

Let $\mathcal{B}$ be a subalgebra of $\mathcal{A}$, and suppose that $\mathcal{A}$ is a normed algebra with norm $\norm{\cdot}$. Suppose further that $\mathcal{A}$ is complete under $\norm{\cdot}$ (i.e., a \emph{Banach algebra} under the norm $\norm{\cdot}$). In this work, we emphasize that the use of the term subalgebra for $\mathcal{B}$ does not imply anything as to whether $\mathcal{B}$ is complete under $\norm{\cdot}$, and if $\mathcal{B}$ happens to be complete, then $\mathcal{B}$ is more appropriately called a \emph{Banach subalgebra} of $\mathcal{A}$.

Denote by $\algI$ the multiplicative identity in $\mathcal{A}$. Given any nonzero element $f$ of $\mathcal{A}$, we interpret $f^0$ as $\algI$. Given $f\in\mathcal{A}$ the linear map $\ad f:\mathcal{A}\rightarrow\mathcal{A}$ is defined by $g\mapsto\lbrack f,g\rbrack$. For the rest of the section suppose $\mathcal{A}$ is an algebra of \emph{operators} on some Hilbert space, i.e., an algebra of bounded linear transformations of the Hilbert space into itself. To avoid confusion, given $f\in\mathcal{A}$, we reserve the term \emph{adjoint} for the operator $f^*$, and we shall not use this word to refer to the map $\ad f$, which we just leave to be in symbols. Furthermore, there should no confusion between the adjoint $f^*$ of the operator $f\in\mathcal{A}$, and the complex conjugate $c^*$ of the scalar $c\in\C$. An operator $f\in\mathcal{A}$ is said to be a \emph{commutator} if there exist operators $l,r\in\mathcal{A}$ such that $f=\lbrack l,r\rbrack$. 

\subsection{Lie polynomial versus commutator}

We clarify the use of the terms Lie polynomial and commutator in the following.  Fix some operators $g_1,g_2,...,g_k\in\mathcal{A}$.  We consider three examples, the first of which is 
\begin{eqnarray}
p_1 & := & \lbrack\lbrack g_1, g_2\rbrack,\lbrack\lbrack g_2, g_1\rbrack, g_1\rbrack\rbrack,\label{p1Lie}\\
& = & \lbrack g_1g_2-g_2g_1\  ,\   g_2g_1^2 - 2g_1g_2g_1 - g_1^2g_2\rbrack .\label{p1comm}
\end{eqnarray}
Since both $g_1g_2-g_2g_1$ and $g_2g_1^2 - 2g_1g_2g_1 - g_1^2g_2$ are operators, by \eqref{p1comm}, we find that $p_1$ is a commutator. By \eqref{p1Lie}, $p_1$ can be expressed in terms of $g_1,g_2$ using only Lie algebra operations in a finite number of steps, and so $p_1$ is a Lie polynomial in $g_1,g_2,...,g_k$. Take any two nonzero operators $g,h\in\mathcal{A}$, and suppose further that $p_2:=\lbrack g,h\rbrack$ is nonzero. Clearly, $p_2$ is a commutator, but is not necessarily a Lie polynomial in $g_1,g_2,...,g_k$, for we have no information if each of $g,h$ can be expressed in terms of $g_1,g_2,...,g_k$ using only Lie algebra operations in a finite number of steps. This information, or more precisely the characterization of Lie polynomials, can be determined by results from the theory of free Lie algebras \cite{Reut}, together with algebraic relations or equations that hold in $\mathcal{A}$. However, we show in this work that by some simple spectral properties of a Hilbert space operator, we are able to characterize Lie polynomials in a particular $q$-deformed Heisenberg algebra. This is because of the interplay of an algebraic structure and a metric space structure that both exist in the algebra. Our final example is $p_3:=g_1+\lbrack g_1, g_2\rbrack+\lbrack\lbrack g_3, g_1\rbrack,g_1\rbrack$, which is a Lie polynomial in $g_1,g_2,...,g_k$, but is not necessarily a commutator. It is possible to use spectral properties of $p_3$ to determine whether $p_3$ is a commutator or not, or to use algebraic relations in $\mathcal{A}$ to prove that $p_3$ is a commutator, but we clarify that the scope of this work does not involve problems of this type. We simply study a finitely-generated Lie subalgebra of a particular $q$-deformed Heisenberg algebra, which includes  linear combinations of commutators, where each of such linear combinations is not necessarily a commutator.

\subsection{Algebraic preliminaries about the $q$-deformed Heisenberg algebra}\label{AlgSec}

Fix a $q\in\F$. The $q$-deformed Heisenberg algebra is the unital associative algebra $\Heisen$ over $\F$ generated by two elements $A,B$ satisfying the relation $AB-qBA=\algI$.  Denote the set of all nonnegative integers by $\N$, and the set of all positive integers by $\Z^+$.  In our computations, the following basis of $\Heisen$ shall be useful.

\begin{proposition}[{\cite[Corollary 4.5]{Hel05}}]\label{BasisProp}
If $q\notin\{0,1\}$, then the following vectors form a basis for $\Heisen$.

\begin{eqnarray}\label{Heisenbasis}
\LieAB^k,\quad\LieAB^k A^l,\quad B^{l}\LieAB^k, \quad (k\in\N,\   l\in\Z^+).
\end{eqnarray}
\end{proposition}

The above basis for $\Heisen$ motivates the following definition, which plays a significant role in the main results in this work.

\begin{definition}\label{preLieDef} We define $\preLie$ as the span of
\begin{eqnarray}
A,\quad B,\quad\LieAB^k,\quad\LieAB^k A^l,\quad B^{l}\LieAB^k, \quad (k,l\in\Z^+),\label{preLiebasis}
\end{eqnarray}
and by $\approxLie$ the span of
\begin{eqnarray}
\algI,\quad A^l,\quad B^{l}, \quad (k,l\in\Z^+\backslash\{1\}).\label{approxLiebasis}
\end{eqnarray}
\end{definition}

Thus, for every $U\in\Heisen$ there exist $V_1\in\preLie$ and $V_2\in\approxLie$ such that $U=V_1+V_2$. Furthermore, $\Heisen$ is equal to the direct sum $\preLie+\approxLie$ of vector spaces. We denote by $\HeisenLie$ the Lie subalgebra of $\Heisen$ generated by $A,B$. 

\begin{proposition}\label{diamondProp} The product of any two basis elements of $\Heisen$ in \eqref{Heisenbasis} is a finite linear combination of \eqref{Heisenbasis}.
\end{proposition}
Proposition~\ref{diamondProp} shall turn out to be crucial in one of our main results because of the finiteness condition it asserts about finite products in the $q$-deformed Heisenberg algebra.  The proof of Proposition~\ref{diamondProp} is based on an algebraic result called the \emph{Diamond Lemma for Ring Theory} \cite[Theorem 1.2]{Berg}, which is a standard tool for determining a Hamel basis of an algebra abstractly defined by generators and relations.

\section{A representation of a $q$-deformed Heisenberg algebra as an algebra of Hilbert space operators}\label{HilbRepSec}
In this section, we recall a representation of $\Heisen$ as an algebra of Hilbert space operators as discussed in \cite[Section 1.3]{Hor}, which turns out to be faithful using a result from \cite{Hel05}. By this representation, we are able to obtain our main results which imply that by simply knowing the spectrum of $B$ as an operator in the said representation, we are able to identify $\HeisenLie$ as precisely $\preLie$, and to identify the image of $\Heisen$ in the Calkin algebra over the relevant Hilbert space. This section culminates with our result about the image of $\Heisen$ in the Calkin algebra. The results about Lie polynomials are in the next section.

Throughout, we set $\F=\C$, the field of complex numbers. Let $\Hilb$ be an infinite-dimensional separable Hilbert space. i.e., The Hilbert space $\Hilb$ is linearly isometric to the sequence space $\HilbSeq$ of all square-summable sequences of scalars \cite[Theorem 1.38]{Fab}. Choose a complete orthonormal basis 
\begin{equation}\{\basvecHilb_n\  :\  n\in\N\}\label{nonketbasis}
\end{equation}
of $\Hilb$, and let $\HilbD\subset\Hilb$ denote the dense subspace spanned by \eqref{nonketbasis}. We can have the interpretation that $\lket n\rket :=\basvecHilb_n$ for any $n\in\N$. The choice of using the notation without the ket is purely for computational convenience. With reference to our notation about unital associative algebras, we denote also by the symbol $\algI$ the identity linear transformation $\HilbD\rightarrow\HilbD$. An operator $V$ on $\Hilb$ is called a \emph{diagonal operator} (with respect to our chosen basis) if there exists a sequence $\left(\alpha_n\right)_{n\in\N}$ of scalars such that $V\basvecHilb_n=\alpha_n\basvecHilb_n$ for any $n$, and in such a case, we call the sequence $\left(\alpha_n\right)_{n\in\N}$ the \emph{diagonal of} $V$. The operator $U$ on $\Hilb$ defined by $U\basvecHilb_n:=\basvecHilb_{n+1}$ is called the \emph{unilateral shift operator}. An operator $L$ on $\Hilb$ is called a \emph{(unilateral) weighted shift operator} if $L=UV$ for some diagonal operator $V$, in which case the sequence of  \emph{weights} of $L$ is defined as the diagonal of $V$. Given $n\in\N$, let $\{n\}_q:=1+q+q^2+\cdots+q^{n-1}$. Observe that if $q\neq 1$, then $\{n\}_q=\frac{1-q^n}{1-q}$.  We define $\newHB:\HilbD\rightarrow\HilbD$ as the linear transformation defined by $\newHB\basvecHilb_n := \sqrt{\{n+1\}_q}\basvecHilb_{n+1}$, and $\newHA:\HilbD\rightarrow\HilbD$ as the linear transformation defined by $\newHA\basvecHilb_n := \sqrt{\{n\}_q}\basvecHilb_{n-1}$ if $n\geq 1$, and $\newHA\basvecHilb_0=0$.

\begin{proposition}\label{routineProp} For any $n\in\N$, let $\commdiag:=\{n\}_q$. If the sequence $\commdiagS_{n\in\N}$  is bounded, then $\newHA$, $\newHB$ are operators  on the Hilbert space $\Hilb$. If this is the case, then we further have:
\begin{enumerate} 
\item\label{ABdiag1} The linear transformation $\newHB$ is a weighted shift operator.
\item\label{ABdiag2} For any $k\in\Z^+$, the operator $\lbrack\newHA ,\newHB\rbrack^k$ is a diagonal operator with diagonal $\left(q^{kn}\right)_{n\in\N}$ for any $k\in\Z^+$.
\end{enumerate}
\end{proposition}
\begin{proof} The proof of this proposition is routine. We only note here that the statement before \ref{ABdiag1} is taken from \cite[Exercise 1.5]{Hor}. The consequences \ref{ABdiag1} and \ref{ABdiag2} can be verified by routine computations. For \ref{ABdiag1}, the required $V$ is the diagonal operator with diagonal $\left(\alpha_n\right)_{n\in\N}$ defined by $\alpha_n:=\sqrt{\{n+1\}_q}$. For \ref{ABdiag2}, the case $k=1$ has been worked out in \cite[Proposition 1.28]{Hor}, and we simply proceed by induction on $k$.
\end{proof}

\begin{lemma}\label{BpmLem} The linear transformations $\newHA$ and $\newHB$ satisfy the relation
\begin{eqnarray}
\newHA\newHB - q \newHB\newHA & = & \algI.\label{qcommREL}
\end{eqnarray}
Furthermore, the operator $\lbrack\newHA,\newHB\rbrack ^k$ is not the zero operator for all $k\in\N$; i.e.,
\begin{eqnarray}
\lbrack\newHA,\newHB\rbrack^k & \neq & 0,\quad\quad(k\in\N).\label{faithfulREL}
\end{eqnarray}
\end{lemma}
\begin{proof} We refer the reader to the facts in \cite[Section 1.3]{Hor}, particularly \cite[Remark 1.34]{Hor}, for \eqref{qcommREL}. To prove \eqref{faithfulREL}, by Proposition~\ref{routineProp}\ref{ABdiag2}, for any $k\in\Z^+$ we have $\lbrack\newHA ,\newHB\rbrack^k\basvecHilb_0=\basvecHilb_0\neq 0$ and so $\lbrack\newHA,\newHB\rbrack ^k$ is not the zero operator.
\end{proof}

Denote by $\OpAlg$ the algebra of all operators on $\Hilb$. By the \emph{Calkin algebra on} $\Hilb$, we mean the algebra $\OpAlg$ modulo the ideal of all compact operators. Suppose that the hypothesis of Proposition~\ref{routineProp} holds, and so $\newHA,\newHB\in\OpAlg$. In view of \eqref{qcommREL}, the multiplicative identity in $\OpAlg$ is expressible in terms of $\newHA$, $\newHB$, and so the subalgebra $\preHeisen$ of $\OpAlg$ generated by $\newHA$, $\newHB$ is unital. The property \eqref{qcommREL} suggests that $\preHeisen$ is a homomorphic image of $\Heisen$, or equivalently, there exists an algebra homomorphism $\Heisen\rightarrow\preHeisen$. Following standard terminology, any such homomorphism of $\Heisen$ onto an algebra of linear transformations of some vector space into itself is called a \emph{representation} of $\Heisen$. As to whether the said representation $\Heisen\rightarrow\preHeisen$ is \emph{faithful} (i.e., injective) or not, the following will be useful.

\begin{proposition}[{\cite[Theorem 6.7]{Hel05}}]\label{faithProp} If $q$ is nonzero and is not a root of unity, then a representation $\varphi$ of $\Heisen$ is faithful if and only if $\lbrack\varphi(A),\varphi(B)\rbrack^k\neq 0$ for all $k\in\N$.
\end{proposition}

\begin{lemma}\label{isoLem} If $\newHA,\newHB\in\OpAlg$, and if $q$ is nonzero and is not a root of unity, then the associative algebra homomorphism $\Heisen\rightarrow\preHeisen$ given by 
\begin{eqnarray}
A\mapsto\newHA,\quad\quad B\mapsto\newHB\label{isoH}
\end{eqnarray}
is an isomorphism.
\end{lemma}
\begin{proof} By \eqref{qcommREL}, the generators $\newHA$, $\newHB$ satisfy the defining relation of $\Heisen$, and so $\preHeisen$ is a homomorphic image of $\Heisen$. By \eqref{faithfulREL} and Proposition~\ref{faithProp}, the representation of $\Heisen$ given by \eqref{isoH} is faithful.
\end{proof}

\begin{remark}\label{isorem}\emph{Throughout, we assume that $q\in\left]0,1\right[$. Then the hypotheses of Proposition~\ref{routineProp} and Lemma~\ref{isoLem} are all satisfied, and so we now identify $A$, $B$ as the operators $\newHA$, $\newHB$, respectively, and $\Heisen$ as the subalgebra of $\OpAlg$ generated by $A,B$. A norm is hence induced in $\Heisen$ which is simply the norm $\norm{\cdot}$ in $\OpAlg$.
}\end{remark}

\begin{proposition}\begin{enumerate}\item\label{pownorm1} The operators $A,B$ are mutually adjoint, and for any $l\in\Z^+$, we have 
\begin{eqnarray} \norm{B^l}=(1-q)^{-\frac{l}{2}}=\norm{A^l}.
\end{eqnarray}
\item\label{pownorm2} For any $k\in\Z^+$, the operator $\LieAB^k$ is Hermitian with norm $1$.
\end{enumerate}
\end{proposition}
\begin{proof}\ref{pownorm1} The proof for the mutual adjointness of $A,B$ is routine. See for instance, \cite[Lemma 1.25]{Hor}. By induction on $l$, we find that the adjoint of $B^l$ is $A^l$, and so $\norm{B^l}=\norm{A^l}$. To compute for $\norm{B^l}$, we recall a formula from \cite[Solution 91]{Halm82} that for the weighted shift $B$ with weights $\alpha_n:=\sqrt{\frac{1-q^{n+1}}{1-q}}$, we have $\norm{B^l}=\sup_n\left|\prod_{i=0}^{l-1}\alpha_{n+i}\right|=(1-q)^{-\frac{l}{2}}$.

\noindent\ref{pownorm2} That $\LieAB=AB-BA$ is Hermitian follows immediately from the fact that $A,B$ are mutually adjoint. By induction, the operator $\LieAB^k$ is Hermitian for any $k\in\Z^+$. Also, since $\LieAB^k$ is a diagonal operator with diagonal $\left(q^{kn}\right)_{n\in\N}$, we recall from \cite[p. 34]{Halm82} that $\norm{\LieAB^k}=\sup_n\left|q^{kn}\right|=1$.
\end{proof}

The actual formula for the action of $\newHA$ and $\newHB$ on the chosen basis elements of $\HilbSeq$ can be found in, among many other sources, the work \cite[Equation (9)]{Chu96}. The work \cite{Chu96} also cites \cite{Sch94} as a source for the representation of $\Heisen$ using mutually adjoint operators. It is just that we use the facts from \cite[Section 1.3]{Hor} because in it, these basic properties of $\newHA$ and $\newHB$ that we need for our goals in this work are nicely presented.

\subsection{Spectral properties of the generators of $\Heisen$}

Our next goal is to compute the spectra of operators in $\Heisen$ that are of significance in the main results of this work. We start with the operator $\LieAB^k$ ($k\in\Z^+$) which has a rather simple spectrum.

\begin{lemma}\label{LieABlem}\begin{enumerate}\item\label{Herm1} For any $k\in\Z^+$, the spectrum of the Hermitian operator $\LieAB^k$ is $\{0\}\cup\{q^{kn}\  :\  n\in\N\}$, where the nonzero elements are precisely all the eigenvalues. The eigenspace for the eigenvalue $q^{kn}$ is the span of $\basvecHilb_n$.
\item\label{Calkin1} If we decompose $\preLie$ into the direct sum $\preLie=\Span\{A,B\}+\derpreLie$ of vector spaces such that a basis for $\derpreLie$ consists of \eqref{preLiebasis} except $A,B$, then each element of $\derpreLie$ is a compact operator.
\end{enumerate}
\end{lemma}
\begin{proof}\ref{Herm1} Since $\LieAB^k$ is a diagonal operator, we recall from \cite[Solution 80]{Halm82} that the scalars in the diagonal, which is precisely the set $\{q^{kn}\  :\  n\in\N\}$, is equal to the set of all eigenvalues of $\LieAB^k$. Also from \cite[Solution 80]{Halm82}, the spectrum of $\LieAB^k$ equals the approximate point spectrum which equals the closure of the set $\{q^{kn}\  :\  n\in\N\}$. It is routine to show that this closure is equal to $\{0\}\cup\{q^{kn}\  :\  n\in\N\}$, which means that the only approximate proper value that is not an eigenvalue is $0$. Consider one of the eigenvalues $q^{kn}$. If $\sum_{m=0}^\infty c_m\basvecHilb_m$ is a nonzero eigenvector for $q^{kn}$, then it is routine to show that $c_m=0$ if $m\neq n$ or nonzero otherwise. Therefore, the eigenspace for $q^{kn}$ is the one-dimensional span of $\basvecHilb_n$.

\noindent\ref{Calkin1} Proving that every basis element of $\derpreLie$ is compact would suffice, but since $\LieAB$ is a multiplicative factor of each such basis element, our task boils down to proving that $\LieAB$ is compact. But by \cite[p. 91]{Halm82}, we simply need to show that the diagonal of $\LieAB$, as a sequence, has limit $0$. By Proposition~\ref{routineProp}\ref{ABdiag2}, this diagonal is $\left(q^{kn}\right)_{n\in\N}$ with $\lim_{n\rightarrow\infty}q^{kn}=0$. Therefore, $\LieAB$ is compact.
\end{proof}

Lemma~\ref{LieABlem}\ref{Calkin1} says, in particular, that $\LieAB^k$ is compact for any $k\in\Z^+$. This is not the case for the generators $A,B$ as we have in the following.

\begin{proposition}\label{FredProp} \begin{enumerate} \item\label{Calkin2} The operator $B$ is not compact but is Fredholm, and the same properties hold for its adjoint $A$.
\item\label{Calkin3} In the Calkin algebra on $\Hilb$, the inverse of $B$ is a scalar multiple of the image of $A$.
\end{enumerate}
\end{proposition}
\begin{proof}\ref{Calkin2} By definition, the weights of $B$ are in the sequence $\left(\alpha_n\right)_{n\in\N}$ with $\alpha_n:=\sqrt{\{n+1\}_q}$ which does not have a limit of $0$. By \cite[p. 91]{Halm82}, this implies that $B$ is not compact. To show that $B$ is Fredholm, we use the relations $\LieAB=AB-BA$ and $AB-qBA=\algI$ to obtain
\begin{eqnarray}
B((1-q)A) & = & \algI + (-1)\LieAB,\label{Fred1}\\
((1-q)A)B & = & \algI + (-q)\LieAB,\label{Fred2}
\end{eqnarray}
where $(1-q)A$ is bounded and the operators $(-1)\LieAB$ and $(-q)\LieAB$ are compact. Then the relations \eqref{Fred1}, \eqref{Fred2} imply that $B$ is Fredholm. By rewriting the left-hand sides of \eqref{Fred1}, \eqref{Fred2} into $((1-q)B)A$ , $A((1-q)B)$, respectively, we find that $A$ is also Fredholm.

\noindent\ref{Calkin3} Let $\pi$ be the canonical map from $\OpAlg$ onto the Calkin algebra on $\Hilb$. The compact operators $-\LieAB$ and $-q\LieAB$ have zero image under $\pi$, and so by \eqref{Fred1}, \eqref{Fred2}, the inverse of $\pi(B)$ is $\pi((1-q)A)=(1-q)\pi(A)$.
\end{proof}

We now compute the spectrum of $B$. Since we have begun to mention the Calkin algebra, it is natural to investigate at least the algebraic structure of the image of $\Heisen$ under the canonical map of $\OpAlg$ onto the Calkin algebra on $\Hilb$, and we show that a very important property of the said algebraic structure can be known simply by knowing the spectrum of $B$. We recall some important facts about the spectra of weighted shifts. If $R$ is a weighted shift with weights in $\left(\alpha_n\right)_{n\in\N}$, we recall the quantities
\begin{eqnarray}
r(R) & = & \lim_{k\in\Z^+}\sup_{n\in\N}\left|\prod_{i=0}^{k-1}\alpha_{n+i}\right|^{\frac{1}{k}},\label{specrad}\\
i(R) & = & \lim_{k\in\Z^+}\inf_{n\in\N}\left|\prod_{i=0}^{k-1}\alpha_{n+i}\right|^{\frac{1}{k}},\label{speci}
\end{eqnarray}
where $r(R)$ is called the \emph{spectral radius} of $R$. The spectral radius of an operator is a rudiment of Operator Theory found in standard texts \cite{Halm82,Halm57}, and is the supremum among  the moduli of all elements of the spectrum of the operator. The expression $i(R)$ is taken from \cite[p. 349]{Ridg} and is useful in the computation of the approximate point spectrum of the weighted shift $R$. We only note here that in \cite{Ridg}, the fixed orthonormal basis of $\Hilb$ is indexed by $\Z^+$.  In our work, we index the orthonormal basis with $\N$ instead of $\Z^+$ as consistent with \cite{Halm82, Hor}, and thus, the formula \eqref{speci} in comparison with that in \cite[p. 349]{Ridg} has been adjusted for the change in indexing set.

\begin{proposition}\label{riProp} $r(B)=(1-q)^{-\frac{1}{2}}=i(B)$.
\end{proposition}
\begin{proof} For each $k\in\Z^+$ and each $n\in\N$, define $$b_n:=b_n(k)=\prod_{i=0}^{k-1}(1-q^{n+i+1})^{\frac{1}{2k}}.$$ We claim that 
\begin{eqnarray}
\sup_{n\in\N}b_n & = & 1,\label{supbn}\\
\inf_{n\in\N}b_n & = & \prod_{i=0}^{k-1} (1-q^{i+1})^{\frac{1}{2k}}.\label{infbn}
\end{eqnarray}
To prove \eqref{supbn}, we note that since $0<q<1$, it is routine to show that $0<b_n< 1$ for any $n$ and so $0<\sup_{n\in\N}b_n\leq 1$. If this inequality is strict, then there exists a positive integer
\begin{eqnarray}
N>-1+\frac{\log(1-(\sup_{n\in\N} b_n)^2)}{\log q},\nonumber
\end{eqnarray}
from which, by routine calculations, we obtain the contradiction $\sup_{n\in\N}b_n<b_N$. Therefore $\sup_{n\in\N}b_n=1$. It is routine to show that the condition $0<q<1$ implies that for any $n\in\N$, we have $b_n<b_{n+1}$, and so $\inf_{n\in\N}b_n=b_0$, where $b_0$ is the right-hand side of \eqref{infbn}. This proves the claim. Recall from the proof of Proposition~\ref{FredProp}\ref{Calkin2} that for $B$ we have the $n$th weight as the positive number $\alpha_n=\sqrt{\frac{1-q^{n+1}}{1-q}}=\left|\alpha_n\right|$, and so by \eqref{specrad} to \eqref{infbn}, we have
\begin{eqnarray}
r(B)  & = & (1-q)^{-\frac{1}{2}}\lim_{k\in\Z^+}\sup_{n\in\N}b_n=(1-q)^{-\frac{1}{2}}\lim_{k\in\Z^+}1=(1-q)^{-\frac{1}{2}},\nonumber\\
i(B)  & = & (1-q)^{-\frac{1}{2}}\lim_{k\in\Z^+}\inf_{n\in\N}b_n,\nonumber\\
& = & (1-q)^{-\frac{1}{2}}\lim_{k\in\Z^+}\prod_{i=0}^{k-1} (1-q^{i+1})^{\frac{1}{2k}}.\label{limiR}
\end{eqnarray}
Define $c_k:=\prod_{i=0}^{k-1} (1-q^{i+1})^{\frac{1}{2k}}$. In view of \eqref{limiR}, we are done if we are able to prove that $\lim_{k\in\Z^+}c_k=1$. By routine calculations and arguments, we have
\begin{eqnarray}
0\quad <\quad c_k & < & 1,\label{boundck}\\
c_{k+1} & = & c_k^{\frac{k}{k+1}}(1-q^{k+1})^{\frac{1}{2(k+1)}},\label{recurrck}
\end{eqnarray}
for all $k\in\Z^+$. Using \eqref{boundck} and \eqref{recurrck}, it is routine to show $c_k<c_{k+1}$ for all $k\in\Z^+$. Thus, $\left(c_k\right)_{k\in\Z^+}$ is a monotonically increasing sequence of positive terms bounded above by $1$. Therefore $\lim_{k\in\Z^+}c_k=1$.
\end{proof}

For simplicity, we assume from this point onward that by any mention of the term \emph{disk}, we mean a disk in the complex plane centered at the origin.

\begin{proposition}\label{specABprop}\begin{enumerate}\item\label{specB} The spectrum of $B$ is the closed disk with radius $\frac{1}{\sqrt{1-q}}$, with the circumference as the approximate point spectrum, and with the point spectrum empty.
\item\label{specA} The spectrum of $A$ is also the closed disk with radius $\frac{1}{\sqrt{1-q}}$ which is equal to the approximate point spectrum, with the interior as the point spectrum.
\end{enumerate}
\end{proposition}
\begin{proof} In this proof, we make use of the traditional notation that for any operator $R$, we denote by $\Pi_0(R)$, $\Pi(R)$, $\Gamma(R)$, and $\Lambda(R)$ the \emph{point spectrum}, \emph{approximate point spectrum}, \emph{compression spectrum}, and \emph{spectrum} of $R$, respectively. We recall from \cite[p. 41]{Halm82}, from \cite[Solution 73]{Halm82}, and from \cite[Theorem 33.3]{Halm57} the following properties of the spectrum of $R$
\begin{eqnarray}
 \Lambda(R) & =  & \Pi(R)\cup\Gamma(R),\label{specR}\\
\Pi_0(R) & \subset & \Pi(R),\label{approxR}\\
\Lambda(R^*) & = & \{c^*\  :\  c\in\Lambda(R)\},\label{specRstar}\\
\Pi_0(R^*) & = & \{c^*\  :\  c\in\Gamma(R)\},\label{pointRstar}\\
\Gamma(R^*) & = & \{c^*\  :\  c\in\Pi_0(R)\}\label{compRstar}.
\end{eqnarray}

\noindent\ref{specB} Similar to the notation in the proof of Proposition~\ref{riProp}, denote the $n$th weight of $B$ by $\alpha_n=\sqrt{\frac{1-q^{n+1}}{1-q}}\neq 0$. Since none of the $\alpha_n$ vanishes, by \cite[Solution 93]{Halm82} we have that $\Pi_0(B)=\emptyset$, and by \cite[Theorem 1]{Ridg}, we have $\Pi(B) = \{c\in\C\  :\  i(B)\leq |c| \leq r(B)\}$, but by Proposition~\ref{riProp}, we further have $\Pi(B)= \left\{c\in\C\  :\   |c|=\frac{1}{\sqrt{1-q}}\right\}$. By \cite[p. 351]{Ridg}, we have $\Lambda(B)=\{c\in\C\  :\   |c|\leq r(B)\}=\left\{c\in\C\  :\   |c|\leq\frac{1}{\sqrt{1-q}}\right\}$. Then by \eqref{specR}, we have $\Gamma(B)=\left\{c\in\C\  :\   |c|<\frac{1}{\sqrt{1-q}}\right\}$. 

\noindent\ref{specA} By \ref{specB}, the set $\Gamma(B)$ has circular symmetry by conjugation, i.e., $\{c^*\  :\  c\in\Gamma(B)\}=\Gamma(B)$, and so by \eqref{pointRstar}, we have $\Pi_0(A)=\Pi_0(B^*)=\Gamma(B)$. Since $\Lambda(B)$ also has circular symmetry,  by \eqref{specRstar}, we have $\Lambda(A)=\Lambda(B^*)=\Lambda(B)$. Since $\Pi_0(B)$ is empty, by \eqref{compRstar}, we have $\Gamma(A)=\Gamma(B^*)=\emptyset$. Then by \eqref{specR}, we have $\Pi(A)=\Lambda(A)$.
\end{proof}

\begin{lemma}\label{noncompactLem} A finite linear combination in $\approxLie$ of the form $$B+c_2B^2+c_3B^3+\cdots+c_kB^k,$$ with $c_k\neq 0$ is not compact.
\end{lemma}
\begin{proof} Let $p(z):=z+c_2z^2+c_3z^3+\cdots+c_kz^k$, and we use the notation in the proof of Proposition~\ref{specABprop} for the spectra and spectral subsets of operators. By the Spectral Mapping Theorem \cite[Theorem 33.1]{Halm57}, we find that $\Lambda(p(B))=p(\Lambda(B)):=\{p(c)\  :\  c\in\Lambda(B)\}$, which by Proposition~\ref{specABprop}\ref{specB} is equal to
\begin{eqnarray}
\left\{p(c)\  :\  |c|\leq\frac{1}{\sqrt{1-q}}\right\}. \label{Bloch}
\end{eqnarray}
We recall that by \cite[p. 96]{Halm82}, if $\Lambda(p(B))$ has a nonzero element that is not in $\Pi_0(p(B))$, then $p(B)$ is not compact and we are done. But also by the Spectral Mapping Theorem and Proposition~\ref{specABprop}\ref{specB}, we have $\Pi_0(p(B))=\emptyset$, and so our task is reduced to proving that \eqref{Bloch} has a nonzero element. To do this, we make use of rudimentary properties of analytic functions. The polynomial $p$ viewed as a function of a complex variable is analytic, and is defined on any region containing the open unit disk in the complex plane. Also, we have $p(0)=0$, and the derivative has the property $p'(0)=1$. By Bloch's Theorem \cite[Theorem XII.1.4]{Conw78}, there exists an open disk $S$ contained in the open unit disk on which $p$ is injective and such that $p(S)$ contains an open disk $Q$ of radius $\frac{1}{72}$. Because the radius of the closed disk $\Lambda(B)$ exceeds $1$, we find that $S\subset\Lambda(B)$, which implies that $p(S)\subset p(\Lambda(B))= \Lambda(p(B))$, which is equal to \eqref{Bloch}, and in the subset $Q$ of $p(S)$, we can certainly find a nonzero complex number. Thus, \eqref{Bloch} has a nonzero element, and this completes the proof.
\end{proof}

\begin{corollary}\label{LaurentCor} The image of $\Heisen$ in the Calkin algebra on $\Hilb$ is algebraically isomorphic to the complex Laurent polynomial algebra in one indeterminate, with basis consisting of all the integral powers of the image of $B$. 
\end{corollary}
\begin{proof} Denote by $\pi$ the canonical map of $\OpAlg$ onto the Calkin algebra on $\Hilb$, and let $D:=\pi(B)$. Let $X\in\Heisen$. Since we are to prove only an algebraic property, assuming that $X$ is a finite linear combination of \eqref{Heisenbasis} would suffice.  By Lemma~\ref{LieABlem}\ref{Calkin1}, we find that $X=Y+Z$ for some compact $Y$ and some $Z$ which is a finite linear combination of powers of $A,B$. Since $Y$ is compact, we have $\pi(Y)=0$, and so $\pi(X)=\pi(Z)$. Recall from Proposition~\ref{FredProp}\ref{Calkin3} that $\pi(A)$ is a scalar multiple of $D^{-1}$ and so $\pi(Z)$ is a finite linear combination of integral powers of $D$. Thus $\pi(X)$, and hence $\pi(\Heisen)$, is contained in the subalgebra of $\OpAlg$ generated by $D,D^{-1}$. To complete the proof, we show that no finite linear combination of integral powers of $D$ is zero. Suppose otherwise. Then for some integers $a,b$ with $a\leq b$ and some scalars $c_a, c_{a+1},\ldots,c_b$ with $c_a$ and $c_b$ both nonzero, we have
\begin{eqnarray}
\sum_{i=a}^bc_iD^i=0.\label{trivlincom}
\end{eqnarray}
By multiplying both sides of \eqref{trivlincom} by $c_a^{-1}D^{1-a}$, we find that the relation
$$D+e_2D^2+e_3D^3+\cdots+e_{1+b-a}D^{1+b-a}=0,$$
holds in the Calkin algebra for some scalars $e_1,e_2,\ldots,e_{1+b-a}$, with $e_{1+b-a}\neq 0$. This implies that in $\Heisen$, the operator
$$B+e_2B^2+e_3B^3+\cdots+e_{1+b-a}B^{1+b-a}$$
with $e_{1+b-a}\neq 0$ is compact, contradicting Lemma~\ref{noncompactLem}. Therefore, $D$ does not satisfy a relation of the form \eqref{trivlincom}, and so there is an associative algebra isomorphism of $\pi(\Heisen)$ onto the algebra $\C[x,x^{-1}]$ of complex Laurent polynomials on the indeterminate $x$ that performs the mapping $D\mapsto x$. Furthermore, the integral powers of $D$ form a basis for $\pi(\Heisen)$ .
\end{proof}

\section{Characterizing Lie polynomial operators by compactness}\label{ApproxSec} 
In this section, we prove our main result that $\HeisenLie$ is precisely $\preLie$. A significant component of the proof is our knowledge about the image of $\Heisen$ in the Calkin algebra over $\Hilb$, which we obtained using information about the spectrum of $B$. 

Given an element $U$ of the Lie algebra $\Heisen$, recall the linear map $\ad U:\Heisen\rightarrow\Heisen$ defined by $V\mapsto\lbrack U,V\rbrack$. Observe that for every $U\in\HeisenLie$, there exist $U_1,U_2,\ldots,U_k, W\in\{A,B\}$ such that 
\begin{eqnarray}
U =(\ad U_1)\circ(\ad U_1)\circ\cdots\circ(\ad U_1)(W).\label{adform}
\end{eqnarray}
Conversely, every element of $\Heisen$ expressible in the form \eqref{adform} is an element of $\HeisenLie$. Take for instance $(\ad A)^2(B)=\lbrack A,\LieAB\rbrack=A\LieAB -\LieAB A$, which by \eqref{adconstruct}, is a scalar multiple of $\LieAB A^2$. Extending this process, every basis vector of $\preLie$ of the form $\LieAB^kA^l$ can be expressed as an element of $\HeisenLie$. Similar routine computations can be used to do the same for those of the form $B^l
\LieAB^k$. Those of the form $\LieAB^k$ can also be expressed as elements of $\HeisenLie$. We summarize in the following lemma the relations that shall hint the reader into the said routine computations.

\begin{lemma}\label{LiePolyLem} Let $k\in\N$ and $l\in\Z^+$. Define $$\baseG(k)  :=  \left(\left(\ad B\right)\circ\left(-\ad \LieAB\right)^{k}\right)\left( \lbrack \LieAB,A\rbrack\right).$$ The  relations 
\begin{eqnarray}
\LieAB^{k+1}A^l & = & -\frac{\left(\left(-\ad \LieAB\right)^k\circ\left(-\ad A\right)^{l+1}\right)\left( B\right)}{(1-q)^l(q^l-1)^k},\label{baseArel}\\
B^l\LieAB^{k+1} & = & \frac{\left(\left(\ad B\right)^{l-1}\circ\left(\ad \LieAB\right)^{k}\right)\left( \lbrack\LieAB,B\rbrack\right)}{(q-1)^{k+1}(1-q^{k+1})^{l-1}},\label{baseBrel}\\
\baseG(k) & = &  q^{-k}(q-1)^{k+1}\{k+1\}_q\LieAB^{k+2}\nonumber\\
& & -q^{1-k}(q-1)^{k+1}\{k\}_q\LieAB^{k+1},\label{baseGrel}
\end{eqnarray}
hold in $\Heisen$. Furthermore, the commutators $\baseG(k)$ satisfy the relation
\begin{eqnarray}
\LieAB^{k+2} = \frac{q^k}{\{k+1\}_q}\sum_{i=0}^k(q-1)^{-(i+1)}\baseG(i) .\label{baseGrel2}
\end{eqnarray}
\end{lemma}
\begin{proof} For each of \eqref{baseArel} to \eqref{baseGrel2}, use induction on the relevant exponents.
\end{proof}

\begin{lemma}\label{derIdLem} The vector space $\derpreLie$ is a Lie ideal of $\Heisen$.
\end{lemma}
\begin{proof} We rewrite the basis elements of $\Heisen$ in \eqref{Heisenbasis} as 
\begin{eqnarray}
\LieAB^k,\quad\LieAB^kA^l, \quad B^l\LieAB^k, & \label{idbasis}\\
\algI, \quad A^l, \quad B^l. & \quad\quad (k,l\in\Z^+) \label{otherbasis}
\end{eqnarray}
Recall from Lemma~\ref{LieABlem}\ref{Calkin1} that the operators in \eqref{idbasis} form a basis for $\derpreLie$. Let $X$ be one of the operators in \eqref{idbasis}, and let $Y$ be any of the operators in \eqref{idbasis},\eqref{otherbasis}. Our goal is to show $\lbrack X,Y\rbrack\in\derpreLie$. By Proposition~\ref{diamondProp}, each of $XY$, $YX$, and hence $\lbrack X,Y\rbrack=XY-YX$, is a finite linear combination of \eqref{idbasis},\eqref{otherbasis}. We denote this finite linear combination by
\begin{eqnarray}
\lbrack X,Y\rbrack = Z + \sum_{i=1}^m c_i A^i + \sum_{j=0}^n e_j B^j,\label{ideq}
\end{eqnarray}
for some  finite linear combination $Z$ of \eqref{idbasis},  and some $m,n\in\Z^+$, where the $c_i$ and $e_j$ are scalars. Thus, $Z\in\derpreLie$. By Lemma~\ref{LieABlem}\ref{Calkin1}, each of the operators $X$, $Z$, and hence also $\lbrack X,Y\rbrack$, is compact. Denote by $\pi$ the canonical map of $\OpAlg$ onto the Calkin algebra on $\Hilb$, and let $D:=\pi(B)$. Then applying $\pi$ on both sides of \eqref{ideq}, we have
\begin{eqnarray}
0 =\sum_{i=1}^m c_i (1-q)^{-i}D^{-i} + \sum_{j=0}^n e_j D^j,\nonumber
\end{eqnarray}
from which we deduce, by Corollary~\ref{LaurentCor}, that all the $c_i$ and $e_j$ are zero. Then \eqref{ideq} becomes $$\lbrack X,Y\rbrack = Z\in\derpreLie.\qedhere$$
\end{proof}

\begin{theorem} The Lie subalgebra $\HeisenLie$ of $\Heisen$ is precisely $\preLie$.
\end{theorem}
\begin{proof} We first show that $\preLie$ is a Lie subalgebra of $\Heisen$. Let $X,Y\in\preLie$. By Lemma~\ref{LieABlem}\ref{Calkin1}, there exist $Z_1,Z_2\in\derpreLie$ and scalars $c_1,c_2,c_3,c_4$ such that $X=Z_1+c_1A+c_2B$ and $Y=Z_2+c_3A+c_4B$. By some routine calculations, we have
\begin{eqnarray}
\lbrack X,Y\rbrack = \lbrack Z_1,Y\rbrack + \lbrack c_1A+c_2B, Z_2\rbrack + (c_1c_4-c_2c_3)\LieAB.\label{subalgeq}
\end{eqnarray}
By Lemma~\ref{derIdLem}, the first two terms in the right-hand side of \eqref{subalgeq} are in $\derpreLie\subset\preLie$, while the third term, by Definition~\ref{preLieDef}, is a basis element of $\preLie$. Thus, $\lbrack X,Y\rbrack\in\preLie$, and so $\preLie$ is a Lie subalgebra of $\Heisen$. By Lemma~\ref{LiePolyLem}, each basis element of $\preLie$ is in $\HeisenLie$. Thus, $\preLie\subset\HeisenLie$. Also, by Definition~\ref{preLieDef}, $A,B\in\preLie$. That is, $\preLie$ is a Lie subalgebra of $\Heisen$ that contains $A,B$ and is contained in $\HeisenLie$. Then by Definition~\ref{LiegenDef}, we have $\preLie=\HeisenLie$.
\end{proof}

Thus, we now have the direct sum decomposition $\Heisen=\HeisenLie+\approxLie$ of vector spaces, where all the Lie polynomials in $A,B$ are in the first direct summand. It turns out that the combination of the algebraic and metric structures in $\Heisen$ has this interesting property that even if the elements of $\approxLie$ are not Lie polynomials in $A,B$, we can find vectors in $\Hilb$ with respect to which, an element of $\approxLie$ acts as a Lie polynomial or Lie series in $A,B$. We show this in the following.

\begin{theorem} For any $X\in\approxLie$ and any $n\in\N$, there exists a sequence $\mathfrak{S}_n=\left(Y_m\right)_{m\in\N}$ consisting of Lie polynomials  in $A,B$ such that $X\basvecHilb_n=\sum_{m=0}^\infty Y_m\basvecHilb_n$.
\end{theorem}
\begin{proof} Without loss of generality, we assume that $X$ is the sum of a summable sequence $\left(X_m\right)_{m\in\N}$ on $\approxLie$ such that for each $m$, we have $X_m=c_m Z$ for some scalar $c_m$ and some $Z\in\{B^l,\  A^l\}$ where $l\in\N\backslash\{1\}$. Fix $n\in\N$ and $k\in\Z^+$. For each $m\in\N$, define $Y_m:=c_mq^{-kn}B^l\LieAB^k$ if $Z=B^l$ for some $l$, or $Y_m:=c_mq^{k(l-n)}\LieAB^kA^l$ if $Z=A^l$ for some $l$. By Lemma~\ref{LiePolyLem}, we find that $Y_m$ is a Lie polynomial in $A,B$ for any $m$. By induction on $k,l$, we can obtain from \eqref{adconstruct} the relation $A^l\LieAB^k = q^{kl}\LieAB^kA^l$ for any $k,l\in\Z^+$. Thus, if $Z=A^l$ , then $X_m-Y_m=c_m(A^l-q^{k(l-n)}\LieAB^kA^l)=c_m(A^l-q^{-kn}A^l\LieAB^k)=c_mA^l(\algI-q^{-kn}\LieAB^k)$. A more straightforward computation gives us $X_m-Y_m=c_mB^l(\algI-q^{-kn}\LieAB^k)$ for the case $Z=B^l$. But since $\basvecHilb_n$ is an eigenvector of $\LieAB^k$ for the eigenvalue $q^{kn}$, we have $(\algI-q^{-kn}\LieAB^k)\basvecHilb_n=0$, and hence 
\begin{eqnarray}
(X_m-Y_m)\basvecHilb_n=0.\label{termzero}
\end{eqnarray}
Let $\varepsilon>0$. Since $X$ is the sum of the summable sequence $\left(X_m\right)$, there exists a finite $J_0\subset\N$ such that for any finite $J\subset\N$ disjoint from $J_0$, we have
\begin{eqnarray}
\norm{X-\sum_{m\in J} X_m}<\varepsilon.\label{summable}
\end{eqnarray}
Let $J$ be an index set as previously described. We have 
\begin{eqnarray}
\norm{X\basvecHilb_n-\sum_{m\in J}Y_m\basvecHilb_n} & \leq & \norm{X\basvecHilb_n-\sum_{m\in J}X_m\basvecHilb_n}+\norm{\sum_{m\in J}X_m\basvecHilb_n-\sum_{m\in J}Y_m\basvecHilb_n},\nonumber\\
& \leq & \norm{X-\sum_{m\in J} X_m}\cdot\norm{\basvecHilb_n}\nonumber\\
& & +\sum_{m\in J}\norm{(X_m-Y_m)\basvecHilb_n},\label{triang}
\end{eqnarray}
where $\norm{\basvecHilb_n}=1$. Then by \eqref{termzero}, \eqref{summable}, and \eqref{triang}, we have $$\norm{X\basvecHilb_n-\sum_{m\in J}Y_m\basvecHilb_n}<\varepsilon.$$ Therefore $X\basvecHilb_n=\sum_{m=0}^\infty Y_m\basvecHilb_n$.
\end{proof}

\begin{corollary} Every element of $\Heisen$ is the sum of a Lie polynomial in $A,B$ and of an operator whose action on each basis element $\basvecHilb_n$ of $\Hilb$ can be approximated by that of a Lie series in $A,B$.
\end{corollary}

We recall here that given a Lie algebra $\mathfrak{K}$, the \emph{derived (Lie) algebra} of $\mathfrak{K}$ is the span of all commutators $\lbrack x,y\rbrack$ for all $x,y\in\mathfrak{K}$. As for the Lie algebra $\preLie$, an immediate consequence of the direct sum decomposition $\preLie=\Span\{A,B\}+\derpreLie$ from Lemma~\ref{LieABlem}\ref{Calkin1} in conjunction with Lemma~\ref{derIdLem} is that the derived algebra of $\preLie$ is $\derpreLie$. By inspection of the basis elements of $\Heisen$ from \eqref{Heisenbasis} in comparison to the defining basis of $\derpreLie$ from the statement of Lemma~\ref{LieABlem}\ref{Calkin1}, we can use Lemma~\ref{LieABlem}\ref{Calkin1} to obtain the following characterization of the derived algebra of $\preLie$ in terms of compactness.

\begin{corollary} The derived algebra of $\preLie$ is precisely the collection of all compact elements of $\Heisen$.
\end{corollary}

\section*{Acknowledgements} The author acknowledges support from the International Mathematical Union (IMU-CDC), from M\"alardalen University, V\"aster\aa s, Sweden, and also from De La Salle University, Manila.

\appendix

\section{Proof of Proposition~\ref{diamondProp}}

Our proof shall make use of the \emph{Diamond Lemma for Ring Theory} \cite[Theorem 1.2]{Berg}, and so we shall recall some related terminology. We introduce each notion within the context of an arbitrary associative algebra $\mathcal{A}$ that is finitely generated, with generators $G_1,G_2,\ldots, G_t$. For each notion that is introduced, we show a particular manifestation in $\Heisen$ of the notion that shall be relevant in the proof. By a \emph{word} on $G_1,G_2,\ldots, G_t$, we mean a finite product $X_1X_2\cdots X_k$ where $X_i\in\{G_1,G_2,\ldots, G_t\}$ for any $i\in\{1,2,\ldots, k\}$. If $k=0$, we interpret the empty product as $\algI$, which we call the \emph{empty word}. By a \emph{subword} of a word $X_1X_2\cdots X_k$ we mean either the empty word or a word of the form $X_aX_{a+1}\cdots X_b$ for some $a,b\in\{1,2,\ldots,k\}$ with $a\leq b$. We now look at what kind of words to consider in $\Heisen$. Let $C:=AB-BA=\LieAB$. Observe that $\Heisen$ is generated by $A,B,C$, and in this proof we shall consider words on $A,B,C$. From \cite[Lemma 3.1]{Hel05}, we have the relations
\begin{eqnarray}
A\LieAB = q\LieAB A,\quad\LieAB B=qB\LieAB.\label{adconstruct}
\end{eqnarray}
Using the defining relation $AB-qBA=\algI$ of $\Heisen$, the relation $C=AB-BA$ and \eqref{adconstruct}, we have the relations 
\begin{eqnarray}
BA & = & \frac{\algI-C}{1-q},\label{redsys1}\\
AB & = & \frac{\algI-qC}{1-q},\label{redsys2}\\
AC & = & qCA,\label{redsys3}\\
CB & = & qBC.\label{redsys4}
\end{eqnarray}
 A set of relations like \eqref{redsys1} to \eqref{redsys4} in which each left-hand side is a word is called a \emph{reduction system}. For later convenience, let us denote the left-hand sides of \eqref{redsys1} to \eqref{redsys4} by $L_1,L_2,L_3, L_4$, respectively, and the corresponding right-hand sides by $g_1,g_2,g_3,g_4$. Suppose that $W_1=f_1$, $W_2=f_2$, $\ldots$ , $W_n=f_n$ is a reduction system in $\mathcal{A}$. For convenience, let us call this arbitrary reduction system $\mathcal{R}$. A word $W$ on $G_1,G_2,\ldots,G_t$ is said to be \emph{irreducible} with respect to $\mathcal{R}$ if none of the left-hand sides $W_1,W_2,\ldots,W_n$ of the relations in the reduction system $\mathcal{R}$ is a subword of $W$. It is routine to show that a word $W$ on $A,B,C$ is irreducible with respect to the reduction system \eqref{redsys1} to \eqref{redsys4} if and only if $W$ is one of the words
\begin{eqnarray}\label{Heisenbasisproof}
C^k,\quad C^k A^l,\quad B^{l}C^k, \quad (k\in\N,\   l\in\Z^+),
\end{eqnarray}
which are exactly the basis elements of $\Heisen$ in \eqref{Heisenbasis}. Let $L,R$ be words on $G_1,G_2,\ldots,G_t$ and let $i\in\{1,2,\ldots,n\}$. By a \emph{reduction} we mean a linear transformation $r=r(L,i,R)$ of $\mathcal{A}$ into itself that fixes every word except the word $LW_iR$, and instead performs the mapping $LW_iR\mapsto Lf_iR$. Observe that since any relation $W_i=f_i$ in $\mathcal{R}$ is an equality, if $\rho$ is any composition of a finite number of reductions, then $\rho(X)=X$ for any $X\in\mathcal{A}$. A word $W$ on $G_1,G_2,\ldots,G_t$ is an \emph{overlap ambiguity} of $\mathcal{R}$ if $W=V_1V_2V_3$ for some nonempty words $V_1, V_2, V_3$ such that each of the words $V_1V_2$ and $V_2V_3$ is among the left-hand sides $W_1,W_2,\ldots,W_n$, say $W_i=V_1V_2$ and $W_j=V_2V_3$ for some $i,j\in\{1,2,\ldots,n\}$. An overlap ambiguity of the form previously shown is said to be \emph{resolvable} if there exist finite compositions $\rho,\rho'$ of reductions such that $\rho(f_iV_3)=\rho'(V_1f_j)$. In here, we consider the identity linear map $\iota$ as the empty composition of reductions. In the reduction system \eqref{redsys1} to \eqref{redsys4}, the words
\begin{eqnarray}
BAB, \quad ABA, \quad BAC, \quad CBA,\quad  ACB,\label{iamb}
\end{eqnarray}
are precisely all the inclusion ambiquities. Consider the inclusion ambiguity $BAB$. It is routine to show that the reductions $\rho:=r(\algI,4,\algI)$ and $\rho':=\iota$ have the property $\rho(g_1B)=\rho'(Bg_2)$. Thus, the inclusion ambiguity $BAB$ is resolvable. It is routine to show that the other inclusion ambiguities in \eqref{iamb} are all resolvable. A word $W$ on $G_1,G_2,\ldots,G_t$ is an \emph{inclusion ambiguity} of $\mathcal{R}$ if $W=U_1W_iU_2$ for some words $U_1,U_2$ and some $i\in\{1,2,\ldots,n\}$ such that $W=W_j$ for some $j\in\{1,2,\ldots,n\}\backslash\{i\}$. An inclusion ambiguity of the form previously described is said to be \emph{resolvable} if there exist finite compositions $\rho,\rho'$ of reductions such that $\rho(U_1f_iU_2)=\rho'(f_j)$. We find that the reduction system \eqref{redsys1} to \eqref{redsys4} has no inclusion ambiguity. We are now ready to invoke the Diamond Lemma \cite[Theorem 1.2]{Berg}, which states that if all the ambiguities of a reduction system are resolvable, then every element of the algebra is expressible as a unique linear combination of irreducible words after the application of a finite number of reductions. Let $f,g$ be any two elements of $\Heisen$ among \eqref{Heisenbasisproof}. Since all the ambiguities of \eqref{redsys1} to \eqref{redsys4} are resolvable, then there exists a composition 
\begin{eqnarray}
\sigma = r_1\circ r_2\circ\cdots\circ r_m\label{sigmared}
\end{eqnarray}
of reductions such that $\sigma(fg)$ is a unique linear combination of \eqref{Heisenbasisproof}. Observe that $fg$ is only a finite sequence of factors from the finite set $\{A,B,C\}$. From the appearance of the right-hand sides of \eqref{redsys1} to \eqref{redsys4}, the reduction $r_m$ shall replace at most one factor of $fg$ by a finite linear combination. The same is true for all the other reductions appearing in \eqref{sigmared}, in which we only have a finite number of such reductions. Therefore, $\sigma(fg)=fg$ is a finite linear combination of \eqref{Heisenbasisproof}. This completes the proof of Proposition~\ref{diamondProp}.


\begin{thebibliography}{1}

\bibitem{Ari76}
M. Arik, D. Coon, {\it Hilbert spaces of analytic functions and generalized coherent states}. J. Math. Phys. {\bf 17} (1976), 524--527.

\bibitem{Berg} G. Bergman, {\it The diamond lemma for ring theory}. Adv. Math. {\bf 29} (1978), 178-218.

\bibitem{Bie89} 
L. Biedenharn, {\it The quantum group $SU_q(2)$ and a $q$-analogue of the boson operators}. 
J. Phy. A {\bf 22} (1989), L873--L878.



\bibitem{Chu96}
W. Chung, A. Klimyk, {\it On position and momentum operators in the $q$-oscillator algebra}. J. Math. Phys. {\bf 37} (1996), 917--932.

\bibitem{Conw78} J. Conway, {\it Functions of one complex variable}. 2nd Edition, Springer-Verlag, 1978.

\bibitem{Erd06} K. Erdmann, M. Wildon {\it Introduction to Lie algebras}. Springer-Verlag, 2006.

\bibitem{Fab} M. Fabian, et al, {\it Functional analysis and infinite-dimensional geometry}. Springer-Verlag, 2001.


\bibitem{Halm82} P. Halmos, {\it A Hilbert space problem book}. 2nd Edition. Springer-Verlag, 1982.

\bibitem{Halm57} P. Halmos, {\it Introduction to Hilbert space and the theory of spectral multiplicity}. Chelsea Publishing Co., 1957.

\bibitem{Hel} L. Hellstr\"{o}m, S. Silvestrov, {\it Commuting elements in $q$-deformed Heisenberg algebras}. World Scientific, 2000.

\bibitem{Hel05} L. Hellstr\"{o}m, S. Silvestrov, {\it Two-sided ideals in $q$-deformed Heisenberg algebras}.  Expo. Math. {\bf 23} (2005), 99-125.

\bibitem{Hor} A. Hora, N. Obata, {\it Quantum probability and spectral analysis of graphs}. Springer-Verlag, 2007.

\bibitem{Lau} N. Laustsen, S. Silvestrov, {\it Heisenberg-Lie commutation relations in Banach algebras}.  Math. Proc. R. Ir. Acad. {\bf 109} (2009), 163-186.

\bibitem{Mac89} 
A. Macfarlane, {\it On $q$-analogues of the quantum harmonic oscillator and the quantum group $SU(2)_q$}. J. Phy. A {\bf 22} (1989), 4581--4588.

\bibitem{Reut} C. Reutenauer, {\it Free Lie algebras}. Oxford Univ. Press, 1993.

\bibitem{Ridg} W. Ridge, {\it Approximate point spectrum of a weighted shift}. Trans. Amer. Math. Soc. {\bf 147} (1970), 349-356.

\bibitem{Sch94} 
K. Schm\"udgen, {\it Integrable operator representations of $\R_q^2$, $X_{q,\gamma}$ and $SL_q(2,\R)$}. Commun. Math. Phys. {\bf 159} (1994), 217--237.


\bibitem{Zhe91}
A. Zhedanov, {\it ``Hidden symmetry" of the Askey-Wilson polynomials}. Theor. Math. Phys. {\bf 89} (1991), 1146–1157.

\end{thebibliography}
\end{document}